\documentclass[12pt,reqno]{amsart}
\usepackage{amsmath,amssymb,amsfonts,amscd,latexsym,amsthm,mathrsfs,verbatim}
\usepackage{shuffle,yfonts}
\usepackage{graphicx}
\newtheorem{lemma}{Lemma}
\newtheorem{theorem}{Theorem}

\newtheorem{proposition}{Proposition}
\theoremstyle{remark}

\newtheorem{conjecture}{\bf Conjecture}

\newtheorem*{acknowledgements}{\bf Acknowledgements}
\let\ol\overline
\let\eps\varepsilon

\let\wt\widetilde
\newcommand\tuffle{\mathop{\rotatebox[origin=c]{180}{$\shuffle$}}}
\newcommand\stuffle{\mathop{\ol{\rotatebox[origin=c]{180}{$\shuffle$}}}}

\newcommand\bzeta{\textswab{Z}}
\newcommand\BD{\mathcal{BD}}
\newcommand\MD{\mathcal{MD}}
\renewcommand{\d}{{\mathrm d}}

\newcommand{\fH}{{\mathfrak{H}}}

\newcommand\bz[1]{\bzeta\bigl[\begin{smallmatrix}#1\end{smallmatrix}\bigr]}
\DeclareMathAlphabet{\mathpzc}{U}{euf}{m}{n}
\def\zq{\mathpzc z_q}
\begin{document}

\title{Multiple $q$-zeta brackets}

\author{Wadim Zudilin}
\address{School of Mathematical and Physical Sciences, The University of Newcastle, Callaghan NSW 2308, AUSTRALIA}
\email{wzudilin@gmail.com}

\date{29 November 2014. \emph{Revised}: 12 March 2014}

\begin{abstract}
The multiple zeta values (MZVs) possess a rich algebraic structure of algebraic relations,
which is conjecturally determined by two different (shuffle and stuffle) products
of a certain algebra of noncommutative words. In a recent work, Bachmann constructed
a $q$-analogue of the MZVs\,---\,the so-called bi-brackets\,---\,for which the two products
are dual to each other, in a very natural way. We overview Bachmann's construction and
discuss the radial asymptotics of the bi-brackets, its links to the MZVs, and related linear
(in)dependence questions of the $q$-analogue.
\end{abstract}


\thanks{The author is supported by Australian Research Council grant DP140101186.}

\maketitle

Apart from the `standard' $q$-model of the multiple zeta values (MZVs),
$$
\zeta_q(s_1,\dots,s_l)
:=(1-q)^{s_1+\dots+s_l}\sum_{n_1>\dots>n_l>0}\frac{q^{(s_1-1)n_1+\dots+(s_l-1)n_l}}{(1-q^{n_1})^{s_1}\dotsb(1-q^{n_l})^{s_l}},
$$
introduced in the earlier works \cite{Br05,OT07}, the different $q$-version
$$
\zq(s_1,\dots,s_l)
:=(1-q)^{s_1+\dots+s_l}\sum_{n_1>\dots>n_l>0}\frac{q^{n_1}}{(1-q^{n_1})^{s_1}\dotsb(1-q^{n_l})^{s_l}}
$$
has received a special attention in the more recent work \cite{CEM13} by
Castillo Medina, Ebrahimi-Fard and Manchon. One of the principal features of the latter $q$-MZVs is that
they are well defined for any collection of integers $s_1,\dots,s_l$, so they do not require regularisation as the former $q$-MZVs
and the MZVs themselves.

In the other recent work \cite{BK13,BK14} Bachmann and K\"uhn introduced and studied
a different $q$-analogue of the MZVs, namely,
\begin{align}
&
[s_1,\dots,s_l]
:=\frac1{(s_1-1)!\dotsb(s_l-1)!}\sum_{\substack{n_1>\dots>n_l>0\\d_1,\dots,d_l>0}}
d_1^{s_1-1}\dotsb d_l^{s_l-1}q^{n_1d_1+\dots+n_ld_l}
\nonumber\\ &\quad
=\frac1{(s_1-1)!\dotsb(s_l-1)!}\sum_{\substack{m_1,\dots,m_l>0\\d_1,\dots,d_l>0}}
d_1^{s_1-1}\dotsb d_l^{s_l-1}q^{(m_1+\dots+m_l)d_1+(m_2+\dots+m_l)d_2+\dots+m_ld_l}.
\label{MB}
\end{align}
The series are generating functions of multiple divisor sums, called (\emph{mono}-)\emph{brackets}, with the $\mathbb Q$-algebra spanned by them denoted by $\MD$.
Note that the $q$-series \eqref{MB} can be alternatively written
$$
[s_1,\dots,s_l]
=\frac1{(s_1-1)!\dotsb(s_l-1)!}\sum_{n_1>\dots>n_l>0}\frac{P_{s_1-1}(q^{n_1})\dotsb P_{s_l-1}(q^{n_l})}{(1-q^{n_1})^{s_1}\dotsb(1-q^{n_l})^{s_l}},
$$
where $P_{s-1}(q)$ are the (slightly modified) Eulerian polynomials:
$$
\frac{P_{s-1}(q)}{(1-q)^s}
=\biggl(q\,\frac{\d}{\d q}\biggr)^{s-1}\frac q{1-q}
=\sum_{d=1}^\infty d^{s-1}q^d.
$$
Since $P_{s-1}(1)=(s-1)!$ it is not hard to verify that
\begin{equation}
\lim_{q\to1^-}(1-q)^{s_1+\dots+s_l}[s_1,\dots,s_l]
=\zeta(s_1,\dots,s_l):=\sum_{n_1>\dots>n_l>0}\frac1{n_1^{s_1}\dotsb n_l^{s_l}}.
\label{lim1}
\end{equation}

More recently \cite{Ba14} Bachmann introduced a more general model of the brackets
\begin{align}
\biggl[\begin{matrix} s_1,\dots,s_l \\ r_1,\dots,r_l \end{matrix}\biggr]
&:=\frac1{r_1!\,(s_1-1)!\dotsb r_l!\,(s_l-1)!}\sum_{\substack{n_1>\dots>n_l>0\\d_1,\dots,d_l>0}}
n_1^{r_1}d_1^{s_1-1}\dotsb n_l^{r_l}d_l^{s_l-1}q^{n_1d_1+\dots+n_ld_l}
\nonumber\\
&\phantom:=\frac1{r_1!\,(s_1-1)!\dotsb r_l!\,(s_l-1)!}\sum_{n_1>\dots>n_l>0}
\frac{n_1^{r_1}P_{s_1-1}(q^{n_1})\dotsb n_l^{r_l}P_{s_l-1}(q^{n_l})}{(1-q^{n_1})^{s_1}\dotsb(1-q^{n_l})^{s_l}},
\label{BB}
\end{align}
which he called \emph{bi-brackets},
in order to describe, in a natural way, the double shuffle relations
of these $q$-analogues of MZVs. Note that the \emph{stuffle} (also known as \emph{harmonic} or \emph{quasi-shuffle}) product
for the both models \eqref{MB} and \eqref{BB} in Bachmann's work comes
from the standard rearrangement of the multiple sums obtained from the term-by-term multiplication of two series.
The other \emph{shuffle} product is then interpreted for the model \eqref{BB} only, as a dual product to the stuffle one via
the \emph{partition duality}. Bachmann further conjectures \cite{Ba14} that the $\mathbb Q$-algebra $\BD$ spanned by the bi-brackets \eqref{BB}
coincides with the $\mathbb Q$-algebra $\MD$.

The goal of this note is to make an algebraic setup for Bachmann's double stuffle relations as well as to demonstrate
that those relations indeed reduce to the corresponding stuffle and shuffle relations in the limit as $q\to1^-$.
We also address the reduction of the bi-brackets to the mono-brackets.

\section{Asymptotics}
\label{sect:A}

The following result allows one to control the asymptotic behaviour of the bi-brackets not only as $q\to1^-$ but also
as $q$ approaches radially a root of unity. This produces an explicit version of the asymptotics used in~\cite{Pup05}
for proving some linear and algebraic results in the case $l=1$.

\begin{lemma}
\label{lem1}
As $q=1-\eps\to1^-$,
$$
\frac1{(s-1)!}\,\frac{P_{s-1}(q^n)}{(1-q^n)^s}
=\frac1{n^s\eps^s}\bigl((1-\eps)F_{s-1}(\eps)+\hat\lambda_s\cdot\eps^s\bigr)-\hat\lambda_s+O(\eps)
$$
where the polynomials $F_k(\eps)\in\mathbb Q[\eps]$ of degree $\max\{0,k-1\}$ are generated by
\begin{align*}
\sum_{k=0}^\infty F_k(\eps)x^k
&=\frac1{1-(1-e^{-\eps x})/\eps}
\\
&= 1 + x + \biggl(-\frac12\eps+1\biggr)x^2 + \biggl(\frac16\eps^2-\eps+1\biggr)x^3
\\ &\qquad
+ \biggl(-\frac1{24}\eps^3+\frac7{12}\eps^2-\frac32\eps+1\biggr)x^4
\\ &\qquad
+ \biggl(\frac1{120}\eps^4-\frac14\eps^3+\frac54\eps^2-2\eps+1\biggr)x^5 + \dotsb
\end{align*}
and
$$
\sum_{s=0}^\infty\hat\lambda_sx^s=-\frac{xe^x}{1-e^x}=1+\frac12x+\sum_{k=1}^\infty\frac{B_{2k}}{(2k)!}x^{2k}
$$
is the generating function of Bernoulli numbers.
\end{lemma}

\begin{proof}
The proof is technical but straightforward.
\end{proof}

By moving the constant term $\hat\lambda_s$ to the right-hand side, we get
\begin{align*}
\frac12+\frac{P_0(q^n)}{1-q^n}
&=\frac1n\cdot\biggl(\eps^{-1}-\frac12\biggr)+O(\eps),
\\
\frac1{12}+\frac{P_1(q^n)}{(1-q^n)^2}
&=\frac1{n^2}\cdot\biggl(\eps^{-2}-\eps^{-1}+\frac1{12}\biggr)+O(\eps),
\\
\frac{P_2(q^n)}{(1-q^n)^3}
&=\frac1{n^3}\cdot\biggl(\eps^{-3}-\frac32\eps^{-2}+\frac12\eps^{-1}\biggr)+O(\eps),
\\
-\frac1{720}+\frac{P_3(q^n)}{(1-q^n)^4}
&=\frac1{n^4}\cdot\biggl(\eps^{-4}-2\eps^{-3}+\frac76\eps^{-2}-\frac16\eps^{-1}-\frac1{720}\biggr)+O(\eps),
\end{align*}
and so on.

\begin{proposition}
\label{prop2}
Assume that $s_1>r_1+1$ and $s_j\ge r_j+1$ for $j=2,\dots,l$.
Then
$$
\biggl[\begin{matrix} s_1,\dots,s_l \\ r_1,\dots,r_l \end{matrix}\biggr]
\sim\frac{\zeta(s_1-r_1,s_2-r_2,\dots,s_l-r_l)}{r_1!\,r_2!\dotsb r_l!}\,\frac1{(1-q)^{s_1+s_2+\dots+s_l}}
\qquad\text{as}\quad q\to1^-,
$$
where $\zeta(s_1,\dots,s_l)$ denotes the standard MZV.
\end{proposition}

Another way to tackle the asymptotic behaviour of the (bi-)brackets is based on the Mellin transform
$$
\varphi(t)\mapsto\wt\varphi(s)=\int_0^\infty\varphi(t)t^{s-1}\d t
$$
which maps
$$
q^{n_1d_1+\dots+n_ld_l}\big|_{q=e^{-t}}\mapsto\frac{\Gamma(s)}{(n_1d_1+\dots+n_ld_l)^s};
$$
see \cite{FS09,Zag06}. Note that the bijective correspondence between the bi-brackets and the zeta functions
$$
\frac{\Gamma(s)}{r_1!\,(s_1-1)!\dotsb r_l!\,(s_l-1)!}\sum_{\substack{n_1>\dots>n_l>0\\d_1,\dots,d_l>0}}
\frac{n_1^{r_1}d_1^{s_1-1}\dotsb n_l^{r_l}d_l^{s_l-1}}{(n_1d_1+\dots+n_ld_l)^s}
$$
can be potentially used for determining the linear relations of the former. A simple illustration is the
linear independence of the depth~1 bi-brackets.

\begin{theorem}
\label{th1}
The bi-brackets $\bigl[\begin{smallmatrix}s_1\\r_1\end{smallmatrix}\bigr]$,
where $0\le r_1<s_1\le n$, $s_1+r_1\le n$, are linearly independent over $\mathbb Q$.
Therefore, the dimension $d_n^{\BD}$ of the $\mathbb Q$-space spanned by all bi-brackets of weight at most~$n$ is
bounded from below by $\lfloor(n+1)^2/4\rfloor\ge n(n+2)/4$.
\end{theorem}

\begin{proof}
Indeed, the functions
\begin{gather*}
\frac{\Gamma(s)}{r_1!\,(s_1-1)!}\sum_{n_1,d_1>0}\frac{n_1^{r_1}d_1^{s_1-1}}{(n_1d_1)^s}
=\Gamma(s)\frac{\zeta(s-s_1+1)\zeta(s-r_1)}{(s_1-1)!\,r_1!},
\\ 
\text{where}\quad 0\le r_1<s_1\le n, \quad s_1+r_1\le n,
\end{gather*}
are linearly independent over~$\mathbb Q$ (because of their disjoint sets of poles at $s=s_1$ and $s=r_1+1$, respectively); thus the corresponding
bi-brackets $\bigl[\begin{smallmatrix}s_1\\r_1\end{smallmatrix}\bigr]$ are $\mathbb Q$-linearly independent as well.
\end{proof}

A similar (though more involved) analysis can be applied
to describe the Mellin transform of the depth~2 bi-brackets; note that it is more easily done for another $q$-model
we introduce further in Section~\ref{sect:D}.

\section{The stuffle product}
\label{sect:T}

Consider the alphabet $Z=\{z_{s,r}:s,r=1,2,\dots\}$ on the double-indexed letters $z_{s,r}$ of the pre-defined weight $s+r-1$.
On $\mathbb QZ$ define the (commutative) product
\begin{align}
z_{s_1,r_1}\diamond z_{s_2,r_2}
&:=\binom{r_1+r_2-2}{r_1-1}\biggl(z_{s_1+s_2,r_1+r_2-1}
\nonumber\\ &\qquad
+\sum_{j=1}^{s_1}(-1)^{s_2-1}\binom{s_1+s_2-j-1}{s_1-j}\lambda_{s_1+s_2-j}z_{j,r_1+r_2-1}
\nonumber\\ &\qquad
+\sum_{j=1}^{s_2}(-1)^{s_1-1}\binom{s_1+s_2-j-1}{s_2-j}\lambda_{s_1+s_2-j}z_{j,r_1+r_2-1}\biggr),
\label{diam}
\end{align}
where
$$
\sum_{s=0}^\infty\lambda_sx^s=-\frac{x}{1-e^x}=1+\sum_{s=1}^\infty\frac{B_s}{s!}x^s
$$
is the generating function of Bernoulli numbers.
Note that $\hat\lambda_s=\lambda_s$ for $s\ge2$, while $\hat\lambda_1=\frac12=-\lambda_1$ in the notation of Section~\ref{sect:A}.

As explained in \cite{BK13} (after the proof of Proposition~2.9),
the product $\diamond$ is also associative. With the help of \eqref{diam} define the stuffle product on the $\mathbb Q$-algebra $\mathbb Q\langle Z\rangle$
recursively by $1\tuffle w=w\tuffle 1:=w$ and
\begin{equation}
aw\tuffle bv
:=a(w\tuffle bv)+b(aw\tuffle v)+(a\diamond b)(w\tuffle v),
\label{tuff}
\end{equation}
for arbitrary $w,v\in\mathbb Q\langle Z\rangle$ and $a,b\in Z$.

\begin{proposition}
\label{prop1}
The evaluation map
\begin{equation}
[\,\cdot\,]\colon z_{s_1,r_1}\dots z_{s_l,r_l}\mapsto\biggl[\begin{matrix} s_1,\dots,s_l \\ r_1-1,\dots,r_l-1 \end{matrix}\biggr]
\label{evamap}
\end{equation}
extended to $\mathbb Q\langle Z\rangle$ by linearity satisfies
$[w\tuffle v]=[w]\cdot[v]$, so that it is a homomorphism of the $\mathbb Q$-algebra $(\mathbb Q\langle Z\rangle,\tuffle)$ onto
$(\BD,\,\cdot\,)$, the latter hence being a $\mathbb Q$-algebra as well.
\end{proposition}

\begin{proof}
The proof follows the lines of the proof of \cite[Proposition 2.10]{BK13} based on the identity
\begin{align*}
&
\frac{n^{r_1-1}P_{s_1-1}(q^n)}{(s_1-1)!\,(r_1-1)!\,(1-q^n)^{s_1}}
\cdot\frac{n^{r_2-1}P_{s_2-1}(q^n)}{(s_2-1)!\,(r_2-1)!\,(1-q^n)^{s_2}}
\\ &\quad
=\binom{r_1+r_2-2}{r_1-1}\frac{n^{r_1+r_2-2}}{(r_1+r_2-2)!}
\biggl(\frac{P_{s_1+s_2-1}(q^n)}{(s_1+s_2-1)!\,(1-q^n)^{s_1+s_2}}
\\ &\quad\qquad
+\sum_{j=1}^{s_1}(-1)^{s_2-1}\binom{s_1+s_2-j-1}{s_1-j}\lambda_{s_1+s_2-j}\frac{P_{j-1}(q^n)}{(j-1)!\,(1-q^n)^j}
\\ &\quad\qquad
+\sum_{j=1}^{s_2}(-1)^{s_1-1}\binom{s_1+s_2-j-1}{s_2-j}\lambda_{s_1+s_2-j}\frac{P_{j-1}(q^n)}{(j-1)!\,(1-q^n)^j}\biggr).
\qedhere
\end{align*}
\end{proof}

Modulo the highest weight, the commutative product \eqref{diam} on $Z$ assumes the form
\begin{equation*}
z_{s_1,r_1}\diamond z_{s_2,r_2}
\equiv\binom{r_1+r_2-2}{r_1-1}z_{s_1+s_2,r_1+r_2-1},
\end{equation*}
so that the stuffle product~\eqref{tuff} reads
\begin{align}
z_{s_1,r_1}w\tuffle z_{s_2,r_2}v
&\equiv z_{s_1,r_1}(w\tuffle z_{s_2,r_2}v)+z_{s_2,r_2}(z_{s_1,r_1}w\tuffle v)
\nonumber\\ &\qquad
+\binom{r_1+r_2-2}{r_1-1}z_{s_1+s_2,r_1+r_2-1}(w\tuffle v)
\label{tuff1}
\end{align}
for arbitrary $w,v\in\mathbb Q\langle Z\rangle$ and $z_{s_1,r_1},z_{s_2,r_2}\in Z$.
If we set $z_s:=z_{s,1}$ and further restrict the product to the subalgebra $\mathbb Q\langle Z'\rangle$, where
$Z'=\{z_s:s=1,2,\dots\}$, then Proposition~\ref{prop2} results in the following statement.

\begin{theorem}[\cite{BK13}]
\label{th:T}
For admissible words $w=z_{s_1}\dots z_{s_l}$ and $v=z_{s_1'}\dotsb z_{s_m'}$ of weight $|w|=s_1+\dots+s_l$ and $|v|=s_1'+\dots+s_m'$, respectively,
$$
[w\tuffle v]\sim(1-q)^{-|w|-|v|}\zeta(w*v)
\qquad\text{as}\quad q\to1^-,
$$
where $*$ denotes the standard stuffle \textup(harmonic\textup) product of MZVs on $\mathbb Q\langle Z'\rangle$.
\end{theorem}

Since $[w]\sim(1-q)^{-|w|}\zeta(w)$, $[v]\sim(1-q)^{-|v|}\zeta(v)$ as $q\to1^-$ and $[w\tuffle v]=[w]\cdot[v]$,
Theorem~\ref{th:T} asserts that the stuffle product \eqref{tuff} of the algebra $\MD$ reduces to the stuffle product of the algebra of MZVs
in the limit as $q\to1^-$. This fact has been already established in~\cite{BK13}.

\section{The duality}
\label{sect:D}

As an alternative extension of the mono-brackets \eqref{MB} we introduce the \emph{multiple $q$-zeta brackets}
\begin{align}
&
\bzeta\biggl[\begin{matrix} s_1,\dots,s_l \\ r_1,\dots,r_l \end{matrix}\biggr]
=\bzeta_q\biggl[\begin{matrix} s_1,\dots,s_l \\ r_1,\dots,r_l \end{matrix}\biggr]
\nonumber\\ &\qquad
:=c\sum_{\substack{m_1,\dots,m_l>0\\d_1,\dots,d_l>0}}
m_1^{r_1-1}d_1^{s_1-1}\dotsb m_l^{r_l-1}d_l^{s_l-1}q^{(m_1+\dots+m_l)d_1+(m_2+\dots+m_l)d_2+\dots+m_ld_l}
\nonumber\\ &\qquad\phantom:
=c\sum_{m_1,\dots,m_l>0}
\frac{m_1^{r_1-1}P_{s_1-1}(q^{m_1+\dots+m_l})m_2^{r_2-1}P_{s_2-1}(q^{m_2+\dots+m_l})\dotsb m_l^{r_l-1}P_{s_l-1}(q^{m_l})}
{(1-q^{m_1+\dots+m_l})^{s_1}(1-q^{m_2+\dots+m_l})^{s_2}\dotsb(1-q^{m_l})^{s_l}}
\label{TB}
\end{align}
where
$$
c=\frac1{(r_1-1)!\,(s_1-1)!\dotsb(r_l-1)!\,(s_l-1)!}.
$$
Then
$$
\biggl[\begin{matrix} s_1 \\ r_1-1 \end{matrix}\biggr]
=\bzeta\biggl[\begin{matrix} s_1 \\ r_1 \end{matrix}\biggr]
\qquad\text{and}\qquad
[s_1,\dots,s_l]
=\biggl[\begin{matrix} s_1,\dots,s_l \\ 0,\dots,0 \end{matrix}\biggr]
=\bzeta\biggl[\begin{matrix} s_1,\dots,s_l \\ 1,\dots,1 \end{matrix}\biggr].
$$

By applying iteratively the binomial theorem in the forms
$$
\frac{(m+n)^{r_1-1}}{(r_1-1)!}\,\frac{n^{r_2-1}}{(r_2-1)!}
=\sum_{j=1}^{r_1+r_2-1}\binom{j-1}{r_2-1}\frac{m^{r_1+r_2-j-1}}{(r_1+r_2-j-1)!}\,\frac{n^{j-1}}{(j-1)!}
$$
and
$$
\frac{(n-m)^{r-1}}{(r-1)!}
=\sum_{i=1}^r(-1)^{r+i}\frac{n^{i-1}}{(i-1)!}\,\frac{m^{r-i}}{(r-i)!}
$$
we see that the $\mathbb Q$-algebras spanned by either \eqref{BB} or \eqref{TB} coincide.
More precisely, the following formulae link the two versions of brackets.

\begin{proposition}
\label{prop2a}
We have
\begin{align*}
\biggl[\begin{array}{cccc} s_1, & s_2, & \dots, & s_l \\ r_1-1, & r_2-1, & \dots, & r_l-1 \end{array}\biggr]
&=\sum_{j_2=1}^{r_1+r_2-1}\binom{j_2-1}{r_2-1}\sum_{j_3=1}^{j_2+r_3-1}\binom{j_3-1}{r_3-1}\dotsb\sum_{j_l=1}^{j_{l-1}+r_l-1}\binom{j_l-1}{r_l-1}
\\ &\qquad\times
\bzeta\biggl[\begin{array}{ccccc} s_1, & s_2, & \dots, & s_{l-1}, & s_l \\ r_1+r_2-j_2, & j_2+r_3-j_3, & \dots, & j_{l-1}+r_l-j_l, & j_l \end{array}\biggr]
\end{align*}
and
\begin{align*}
&
\bzeta\biggl[\begin{matrix} s_1,\dots,s_l \\ r_1,\dots,r_l \end{matrix}\biggr]
=\sum_{i_1=1}^{r_1}\sum_{i_2=1}^{r_2}\dotsb\sum_{i_{l-1}=1}^{r_{l-1}}
(-1)^{r_1+\dots+r_{l-1}-i_1-\dotsb-i_{l-1}}
\\ &\qquad\times
\binom{r_1-i_1+i_2-1}{r_1-i_1}\dotsb
\binom{r_{l-2}-i_{l-2}+i_{l-1}-1}{r_{l-2}-i_{l-2}}\binom{r_{l-1}-i_{l-1}+r_l-1}{r_{l-1}-i_{l-1}}
\\ &\qquad\times
\biggl[\begin{array}{ccccc} s_1, & s_2, & \dots, & s_{l-1}, & s_l \\ i_1-1, & r_1-i_1+i_2-1, & \dots, & r_{l-2}-i_{l-2}+i_{l-1}-1, & r_{l-1}-i_{l-1}+r_l-1 \end{array}\biggr].
\end{align*}
\end{proposition}

Proposition~\ref{prop2a} allows us to construct an isomorphism $\varphi$ of the two $\mathbb Q$-algebras $\mathbb Q\langle Z\rangle$ with two evaluation maps
$[\,\cdot\,]$ and $\bzeta[\,\cdot\,]$,
$$
\bzeta[z_{s_1,r_1}\dots z_{s_l,r_l}]=\bzeta\biggl[\begin{matrix} s_1,\dots,s_l \\ r_1,\dots,r_l \end{matrix}\biggr],
$$
such that
$$
[w]=\bzeta[\varphi w] \qquad\text{and}\qquad
\bzeta[w]=[\varphi^{-1}w].
$$
Note however that the isomorphism breaks the simplicity of defining the stuffle product $\tuffle$ from Section~\ref{sect:T}.

Another algebraic setup can be used for the $\mathbb Q$-algebra $\mathbb Q\langle Z\rangle$ with evaluation $\bzeta$.
We can recast it as the $\mathbb Q$-subalgebra $\fH^0:=\mathbb Q+x\fH y$ of the $\mathbb Q$-algebra $\fH:=\mathbb Q\langle x,y\rangle$
by setting $\bzeta[1]=1$ and
$$
\bzeta[x^{s_1}y^{r_1}\dots x^{s_l}y^{r_l}]=\bzeta\biggl[\begin{matrix} s_1,\dots,s_l \\ r_1,\dots,r_l \end{matrix}\biggr].
$$
The depth (or length) is defined as the number of appearances of the subword $xy$, while the weight is the number of letters minus the length.

\begin{proposition}[Duality]
\label{prop3}
$$
\bzeta\biggl[\begin{matrix} s_1,s_2,\dots,s_l \\ r_1,r_2,\dots,r_l \end{matrix}\biggr]
=\bzeta\biggl[\begin{matrix} r_l,r_{l-1},\dots,r_1 \\ s_l,s_{l-1},\dots,s_1 \end{matrix}\biggr].
$$
\end{proposition}

\begin{proof}
This follows from the rearrangement of the summation indices:
$$
\sum_{i=1}^ld_i\sum_{j=i}^lm_j
=\sum_{i=1}^ld_i'\sum_{j=i}^lm_j'
$$
where $d_i'=m_{l+1-i}$ and $m_j'=d_{l+1-j}$.
\end{proof}

Denote by $\tau$ the anti-automorphism of the algebra $\fH$, interchanging
$x$ and $y$; for example, $\tau(x^2yxy)=xyxy^2$. Clearly, $\tau$ is an involution
preserving both the weight and depth, and it is also an automorphism of the subalgebra $\fH^0$.
The duality can be then stated as
\begin{equation}
\bzeta[\tau w]=\bzeta[w]
\qquad\text{for any}\quad w\in\fH^0.
\label{eq:dual}
\end{equation}
We also extend $\tau$ to $\mathbb Q\langle Z\rangle$ by linearity.

The duality in Proposition~\ref{prop3} is exactly the partition duality given earlier by Bachmann for the model~\eqref{BB}.

\section{The dual stuffle product}
\label{sect:S}

We can now introduce the product which is dual to the stuffle one.
Namely, it is the duality composed with the stuffle product and, again, with the duality:
\begin{equation}
w\stuffle v:=\varphi^{-1}\tau(\tau\varphi w\tuffle\tau\varphi v)
\qquad\text{for}\quad w,v\in\mathbb Q\langle Z\rangle.
\label{shuff}
\end{equation}
It follows then from Propositions \ref{prop1} and \ref{prop3} that

\begin{proposition}
\label{prop4}
The evaluation map \eqref{evamap} on $\mathbb Q\langle Z\rangle$ satisfies
$[w\stuffle v]=[w]\cdot[v]$, so that it is also a homomorphism of the $\mathbb Q$-algebra $(\mathbb Q\langle Z\rangle,\stuffle)$ onto
$(\BD,\,\cdot\,)$.
\end{proposition}

Note that \eqref{tuff1} is also equivalent to the expansion from the right \cite[Theorem~9]{Zud03}:
\begin{align}
wz_{s_1,r_1}\tuffle vz_{s_2,r_2}
&\equiv(w\tuffle vz_{s_2,r_2})z_{s_1,r_1}+(wz_{s_1,r_1}\tuffle v)z_{s_2,r_2}
\nonumber\\ &\qquad
+\binom{r_1+r_2-2}{r_1-1}(w\tuffle v)z_{s_1+s_2,r_1+r_2-1}.
\label{tuff2}
\end{align}

The next statement addresses the structure of the dual stuffle product \eqref{shuff}
for the words over the sub-alphabet $Z'=\{z_s=z_{s,1}:s=1,2,\dots\}\subset Z$. Note that
the words from $\mathbb Q\langle Z'\rangle$ can be also presented as the words from $\mathbb Q\langle x,xy\rangle$
necessarily ending with $xy$.

\begin{proposition}
\label{prop4a}
Modulo the highest weight and depth,
\begin{equation}
aw\stuffle bv
\equiv a(w\stuffle bv)+b(aw\stuffle v)
\label{shuff1}
\end{equation}
for arbitrary words $w,v\in\mathbb Q+\mathbb Q\langle x,xy\rangle xy$ and $a,b\in\{x,xy\}$.
\end{proposition}

\begin{proof}
First note that restricting \eqref{tuff2} further modulo the highest depth implies
\begin{align*}
wz_{s_1,r_1}\tuffle vz_{s_2,r_2}
&\equiv(w\tuffle vz_{s_2,r_2})z_{s_1,r_1}+(wz_{s_1,r_1}\tuffle v)z_{s_2,r_2},
\\ \intertext{and that we also have}
wz_{s_1,r_1+1}\tuffle vz_{s_2,r_2}
&\equiv(wz_{s_1,r_1}\tuffle vz_{s_2,r_2})y+(wz_{s_1,r_1+1}\tuffle v)z_{s_2,r_2},
\\
wz_{s_1,r_1+1}\tuffle vz_{s_2,r_2+1}
&\equiv(wz_{s_1,r_1}\tuffle vz_{s_2,r_2+1})y+(wz_{s_1,r_1+1}\tuffle vz_{s_2,r_2})y.
\end{align*}
The relations already show that
\begin{equation}
wa'\tuffle vb'
\equiv(w\tuffle vb')a'+(wa'\tuffle v)b'
\label{tuff3}
\end{equation}
for arbitrary words $w,v\in\mathbb Q+\mathbb Q\langle Z\rangle$ and $a',b'\in Z\cup\{y\}$,
where
$$
z_{s_1,r_1}\dots z_{s_{l-1},r_{l-1}}z_{s_l,r_l}y=z_{s_1,r_1}\dots z_{s_{l-1},r_{l-1}}z_{s_l,r_l+1}.
$$

Secondly note that the isomorphism $\varphi$ of Proposition~\ref{prop2a} acts trivially on the words from $\mathbb Q\langle Z'\rangle$.
Therefore, applying $\tau\varphi$ to the both sides of~\eqref{shuff} and extracting the homogeneous part of the result
corresponding to the highest weight and depth we arrive at
\begin{equation*}
\tau(w\stuffle v)\equiv\tau w\tuffle\tau v
\qquad\text{for all}\quad w,v\in\mathbb Q\langle Z'\rangle.
\end{equation*}
Denoting
$$
\ol a=\tau a=\begin{cases}
y & \text{if $a=x$}, \\
xy & \text{if $a=xy$},
\end{cases}
$$
and using \eqref{tuff3} we find out that
\begin{align*}
\tau(aw\stuffle bv)
&\equiv\tau(aw)\tuffle\tau(bv)
\equiv(\tau w)\ol a\tuffle(\tau v)\ol b
\\
&\equiv(\tau w\tuffle(\tau v)\ol b)\ol a+((\tau w)\ol a\tuffle\tau v)\ol b
\\
&\equiv(\tau w\tuffle\tau(bv))\ol a+(\tau(aw)\tuffle\tau v)\ol b
\equiv(\tau(w\stuffle bv))\ol a+(\tau(aw\stuffle v))\ol b
\\
&\equiv\tau(a(w\stuffle bv)+b(aw\stuffle v)),
\end{align*}
which implies the desired result.
\end{proof}

\begin{theorem}
\label{th:S}
For admissible words $w=z_{s_1}\dots z_{s_l}$ and $v=z_{s_1'}\dotsb z_{s_m'}$ of weight $|w|=s_1+\dots+s_l$ and $|v|=s_1'+\dots+s_m'$, respectively,
$$
[w\stuffle v]\sim(1-q)^{-|w|-|v|}\zeta(w\shuffle v)
\qquad\text{as}\quad q\to1^-,
$$
where $\shuffle$ denotes the standard shuffle product of MZVs on $\mathbb Q\langle Z'\rangle$.
\end{theorem}

\begin{proof}
Because both $\varphi$ and $\tau$ respect the weight,
Proposition~\ref{prop4a} shows that the only terms that can potentially interfere with the asymptotic behaviour as $q\to1^-$
correspond to the same weight but lower depth. However, according to \eqref{shuff} and \eqref{tuff2}, the `shorter' terms
do not belong to $\mathbb Q\langle Z'\rangle$, that is, they are linear combinations of the monomials
$z_{q_1,r_1}\dots z_{q_n,r_n}$ with $r_1+\dots+r_n=l+m>n$, hence $r_j\ge2$ for at least one $j$.
The latter circumstance and Proposition~\ref{prop2} then imply
\begin{equation*}
\lim_{q\to1^-}(1-q)^{|w|+|v|}[z_{q_1,r_1}\dots z_{q_n,r_n}]=0.
\qedhere
\end{equation*}
\end{proof}

Theorem~\ref{th:S} asserts that the dual stuffle product \eqref{shuff} restricted from $\BD$ to the subalgebra $\MD$
reduces to the shuffle product of the algebra of MZVs in the limit as $q\to1^-$. This result is implicitly stated in~\cite{Ba14}.
More is true: using \eqref{tuff1} and Proposition~\ref{prop4a} we obtain

\begin{theorem}
\label{th:TS}
For two words $w=z_{s_1}\dots z_{s_l}$ and $v=z_{s_1'}\dotsb z_{s_m'}$, not necessarily admissible,
$$
[w\tuffle v-w\stuffle v]\sim(1-q)^{-|w|-|v|}\zeta(w*v-w\shuffle v)
\qquad\text{as}\quad q\to1^-,
$$
whenever the MZV on the right-hand side makes sense.
\end{theorem}

In other words, the $q$-zeta model of bi-brackets provides us with a (far reaching) regularisation of the MZVs:
the former includes the extended double shuffle relations as the limiting $q\to1^-$ case.

\begin{conjecture}[{Bachmann \cite{Ba14}}]
\label{conj1}
The resulting double stuffle (that is, stuffle and dual stuffle)
relations exhaust all the relations between the bi-brackets.
Equivalently (and simpler), the stuffle relations and the duality exhaust all the relations between the bi-brackets.
\end{conjecture}

We would like to point out that the duality $\tau$ from Section~\ref{sect:D} also exists for the algebra of MZVs \cite[Section~6]{Zud03}.
However the two dualities are not at all related: the limiting $q\to1^-$ process squeezes the appearances of $x$ before $y$ in
the words $x^{s_1}yx^{s_2}y\dots x^{s_l}y$, so that they become $x^{s_1-1}yx^{s_2-1}y\dots x^{s_l-1}y$. Furthermore, the duality of MZVs
respects the shuffle product: the dual shuffle product coincides with the shuffle product itself. On the other hand,
the dual stuffle product of MZVs is very different from the stuffle (and shuffle) products. It may be an interesting
problem to understand the double stuffle relations of the algebra of MZVs.

\section{Reduction to mono-brackets}
\label{sect:R}

In this final section we present some observations towards another conjecture of Bachmann about the coincidence of the $\mathbb Q$-algebras of bi- and mono-brackets.

\begin{conjecture}[Bachmann]
\label{conj2}
$\MD=\BD$.
\end{conjecture}

Based on the representation of the elements from $\BD$ as the polynomials from $\mathbb Q\langle x,y\rangle$
(see also the last paragraph of Section~\ref{sect:S}), we can loosely interpret this conjecture for the algebra of MZVs
as follows: all MZVs lie in the $\mathbb Q$-span of
$$
\zeta(s_1,s_2,\dots,s_l)=\zeta(x^{s_1-1}yx^{s_2-1}y\dots x^{s_l-1}y)
$$
with all $s_j$ to be at least~2 (so that there is no appearance of $y^r$ with $r\ge2$). The latter statement
is already known to be true: Brown~\cite{Bro12} proves that one can span the $\mathbb Q$-algebra of MZVs by the set with all $s_j\in\{2,3\}$.

\medskip
In what follows we analyse the relations for the model \eqref{TB}, because it makes simpler keeping track of the duality relation.
We point out from the very beginning that the linear relations given below are all experimentally found
(with the check of 500 terms in the corresponding $q$-expansions) but we believe that it is possible to establish
them rigorously using the double stuffle relations given above.

The first presence of the $q$-zeta brackets that are not reduced to ones from $\MD$ by the duality relation happens in weight~3.
It is $\bz{2\\2}$ and we find out that
$$
\bz{2\\2}
=\tfrac12\bz{2\\1}+\bz{3\\1}-\bz{2,1\\1,1}.
$$
There are 34 totally $q$-zeta brackets of weight up to~4,
\begin{gather*}
\bz{}^*, \; \bz{1\\1}^*, \; \bz{2\\1}=\bz{1\\2}, \; \bz{2\\2}^*, \; \bz{3\\1}=\bz{1\\3}, \; \bz{3\\2}=\bz{2\\3}, \; \bz{4\\1}=\bz{1\\4},
\\
\bz{1,1\\1,1}^*, \; \bz{2,1\\1,1}=\bz{1,1\\1,2}, \; \bz{1,2\\1,1}=\bz{1,1\\2,1}, \;
\bz{2,1\\2,1}=\bz{1,2\\1,2}, \; \bz{2,1\\1,2}^*, \; \bz{1,2\\2,1}^*,
\\
\bz{2,2\\1,1}=\bz{1,1\\2,2}, \; \bz{3,1\\1,1}=\bz{1,1\\1,3}, \; \bz{1,3\\1,1}=\bz{1,1\\3,1},
\\
\bz{1,1,1\\1,1,1}^*, \; \bz{2,1,1\\1,1,1}=\bz{1,1,1\\1,1,2}, \;
\bz{1,2,1\\1,1,1}=\bz{1,1,1\\1,2,1}, \; \bz{1,1,2\\1,1,1}=\bz{1,1,1\\2,1,1}, \;
\bz{1,1,1,1\\1,1,1,1}^*,
\end{gather*}
where the asterisk marks the self-dual ones. Only 21 of those listed are not dual-equivalent, and only five of the latter are not reduced to
the $q$-zeta brackets from $\MD$; besides the already mentioned $\bz{2\\2}$ these are $\bz{3\\2}$, $\bz{2,1\\2,1}$, $\bz{2,1\\1,2}$ and $\bz{1,2\\2,1}$. We find out that
\begin{align*}
\bz{3\\2}
&=\tfrac14\bz{2\\1}+\tfrac32\bz{4\\1}-2\bz{2,2\\1,1},
\displaybreak[2]\\
\bz{2,1\\2,1}
&=\bz{2,1\\1,1}+\tfrac12\bz{1,2\\1,1}-\bz{2,2\\1,1}+\bz{1,3\\1,1}-\bz{2,1,1\\1,1,1}-\bz{1,2,1\\1,1,1},
\displaybreak[2]\\
\bz{2,1\\1,2}
&=-\tfrac12\bz{2,1\\1,1}-\tfrac12\bz{1,2\\1,1}+2\bz{2,2\\1,1}+\bz{3,1\\1,1}-\bz{1,3\\1,1}+\bz{1,2,1\\1,1,1},
\displaybreak[2]\\
\bz{1,2\\2,1}
&=-\bz{2,1\\1,1}+2\bz{2,2\\1,1}+\bz{2,1,1\\1,1,1},
\end{align*}
and there is one more relation in this weight between the $q$-zeta brackets from $\MD$:
$$
\tfrac13\bz{2\\1}-\bz{3\\1}+\bz{4\\1}-2\bz{2,2\\1,1}+2\bz{3,1\\1,1}=0.
$$
The computation implies that the dimension $d_4^{\BD}$ of the $\mathbb Q$-space spanned by all multiple $q$-zeta brackets of weight not more than~4 is equal
to the dimension $d_4^{\MD}$ of the $\mathbb Q$-space spanned by all such brackets from $\MD$ and that both are equal to~15.
A similar analysis demonstrates that
$$
d_5^{\BD}=d_5^{\MD}=28 \quad\text{and}\quad d_6^{\BD}=d_6^{\MD}=51,
$$
and it seems less realistic to compute and verify that $d_n^{\BD}=d_n^{\MD}$ for $n\ge7$ though Conjecture~\ref{conj2}
and \cite[Conjecture (5.4)]{BK13} support
\begin{align*}
\sum_{n=0}^\infty d_n^{\MD}x^n
&\overset?=\frac{1-x^2+x^4}{(1-x)^2(1-2x^2-2x^3)}
\\
&=1+2x+4x^2+8x^3+15x^4+28x^5+51x^6+92x^7
\\ &\quad
+165x^8+294x^9+523x^{10}+O(x^{11}).
\end{align*}
We can compare this with the count $c_n^{\MD}$ and $c_n^{\BD}$ of all mono- and bi-brackets of weight $\le n$,
$$
\sum_{n=0}^\infty c_n^{\MD}x^n=\frac1{1-2x}
\quad\text{and}\quad
\sum_{n=0}^\infty c_n^{\BD}x^n=\frac{1-x}{1-3x+x^2}=\sum_{n=0}^\infty F_{2n}x^n,
$$
where $F_n$ denotes the Fibonacci sequence.

In addition, we would like to point out one more expectation for the algebra of (both mono- and bi-) brackets, which is not shared
by other $q$-models of MZVs: all linear (hence algebraic) relations between them seem to be over $\mathbb Q$, not over $\mathbb C(q)$.

\begin{conjecture}
\label{conj3}
A collection of \textup(bi-\textup)brackets is linearly dependent over $\mathbb C(q)$ if and only if
it is linearly dependent over~$\mathbb Q$.
\end{conjecture}

\begin{acknowledgements}
I have greatly benefited from discussing this work with Henrik Bachmann, Kurusch Ebrahimi-Fard, Herbert Gangl and Ulf K\"uhn\,---\,it is my pleasure
to thank them for numerous clarifications, explanations and hints. I thank the three anonymous referees of the journal
for pointing out some typos in the preliminary version and helping to improve the exposition.
I would also like to acknowledge that a part of this research was undertaken in ICMAT\,---\,Institute of Mathematical Sciences
(Universidad Aut\'onoma de Madrid, Spain)
during the Research Trimester on \emph{Multiple Zeta Values, Multiple Polylogarithms, and Quantum Field Theory}
(September--December 2014).
\end{acknowledgements}



\begin{thebibliography}{99}

\bibitem{Ba14}
\textsc{H.~Bachmann},
Generating series of multiple divisor sums and other interesting $q$-series,
\emph{Talk slides} (University of Bristol, 8 July 2014).

\bibitem{BK13}
\textsc{H.~Bachmann} and \textsc{U.~K\"uhn},
The algebra of generating functions for multiple divisor sums and applications to multiple zeta values,
\emph{Preprint} (2013 \& 2014), \texttt{arXiv:\,1309.3920v2 [math.NT]}.

\bibitem{BK14}
\textsc{H.~Bachmann} and \textsc{U.~K\"uhn},
A short note on a conjecture of Okounkov about a $q$-analogue of multiple zeta values,
\emph{Preprint} (2014), \texttt{arXiv:\,1407.6796 [math.NT]}.

\bibitem{Br05}
\textsc{D.\,M. Bradley},
Multiple $q$-zeta values,
\emph{J. Algebra} \textbf{283}:2 (2005), 752--798.

\bibitem{Bro12}
\textsc{F.\,C.\,S. Brown},
Tate motives over $\mathbb Z$,
\emph{Ann. of Math.} (2) \textbf{175}:2 (2012), 949--976.

\bibitem{CEM13}
\textsc{J. Castillo Medina}, \textsc{K. Ebrahimi-Fard} and \textsc{D. Manchon},
Unfolding the double shuffle structure of $q$-multiple zeta values,
\emph{Bull. Austral. Math. Soc.} (to appear);
\emph{Preprint} (2014), \texttt{arXiv:\,1310.1330v3 [math.NT]}.

\bibitem{FS09}
\textsc{P. Flajolet} and \textsc{R. Sedgewick},
\emph{Analytic combinatorics}
(Cambridge Univ. Press, Cambridge 2009).

\bibitem{OT07}
\textsc{J. Okuda} and \textsc{Y.~Takeyama},
On relations for the multiple $q$-zeta values,
\emph{Ramanujan J.} \textbf{14}:3 (2007), 379--387.

\bibitem{Pup05}
\textsc{Yu.\,A.~Pupyrev},
Linear and algebraic independence of $q$-zeta values,
\emph{Math. Notes} \textbf{78}:4 (2005), 563--568.

\bibitem{Zag06}
\textsc{D.~Zagier},
The Mellin transform and other useful analytic techniques,
Appendix to E.~Zeidler, \emph{Quantum Field Theory I: Basics in Mathematics and Physics. A Bridge Between Mathematicians and Physicists}
(Springer-Verlag, Berlin--Heidelberg--New York 2006), 305--323;
\texttt{http://people.mpim-bonn.mpg.de/zagier/files/tex/MellinTransform/fulltext.pdf}.

\bibitem{Zud03}
\textsc{W.~Zudilin},
Algebraic relations for multiple zeta values,
\emph{Uspekhi Mat. Nauk} \textbf{58}:1 (2003), 3--32;
English transl.,
\emph{Russian Math. Surveys} \textbf{58}:1 (2003), 1--29.

\end{thebibliography}
\end{document}